\documentclass[12pt]{article}

\usepackage[top=20truemm,bottom=20truemm,left=22truemm,right=22truemm]{geometry}
\usepackage{amsthm}
\usepackage{amssymb}
\usepackage{amsmath}
\usepackage{multicol}
\usepackage{color}
\usepackage{graphicx}
\usepackage{amscd}
\usepackage[all]{xy}
\usepackage{comment}
\usepackage{marvosym}
\usepackage{tikz}
\usetikzlibrary{cd}
\usepackage{tasks}

\usepackage{color, hyperref}
\hypersetup{
    colorlinks=false,
    citebordercolor=green,
    linkbordercolor=red,
    urlbordercolor=cyan,
}

\theoremstyle{plain}
\newtheorem{theorem}{Theorem}[section]

\newtheorem{proposition}[theorem]{Proposition}
\newtheorem{lemma}[theorem]{Lemma}
\newtheorem{corollary}[theorem]{Corollary}

\newtheorem{claim}[theorem]{Claim}
\newtheorem*{claim*}{Claim}

\theoremstyle{definition}
\newtheorem{definition}[theorem]{Definition}

\newtheorem{example}[theorem]{Example}

\newtheorem*{setup*}{Setup}

\theoremstyle{remark}
\newtheorem{remark}[theorem]{Remark}

\newcommand{\Spec}{\mathrm{Spec}}

\newcommand{\sht}{\mathrm{ht}^{\mathrm{s}}}

\newcommand{\Hcr}{H_{\mathrm{cris}}}

\newcommand{\Hcris}{H_{cris}}

\newcommand{\F}{\mathbb{F}}

\renewcommand{\P}{\mathbb{P}}
\newcommand{\Q}{\mathbb{Q}}

\newcommand{\Z}{\mathbb{Z}}

\renewcommand{\O}{\mathcal{O}}

\newcommand{\cHom}{\mathcal{H}om}

\title{Quasi-F-split and Hodge-Witt}
\author{Fuetaro Yobuko}
\date{}

\begin{document}

\maketitle

\section{Introduction}

Let $p$ be a prime number.
This paper treats positive characteristic $p>0$ algebraic geometry.
Quasi-$F$-splitting is an extension of the notion of \textit{$F$-splitting}, introduced by Mehta-Ramanathan \cite{mehta-ramnathan}.
Quasi-$F$-splitting is introduced in \cite{yobuko19} and recently studied in \cite{kttwyy1} and \cite{kttwyy2} from the point of view of birational geometry.
This notion is influenced by the \textit{Artin-Mazur height} of Calabi-Yau varieties, which takes values in positive integers or infinity.
It is known that the Artin-Mazur height one is equivalent to $F$-splitting and such a variety is called ordinary at least its dimension $\leq 2$.
In \cite{yobuko19}, the author shows that the Artin-Mazur height being finite is equivalent to quasi-$F$-splitting for Calabi-Yau varieties.

In this note, we investigate a theory of (quasi-)$F$-splitting in the form of
\begin{align*}
    \text{$F$-split}:\text{quasi-$F$-split}=\text{ordinary}:\text{Hodge-Witt}.
\end{align*}
Here, ordinary means \textit{ordinary in the sense of Bloch-Kato}.
See \ref{ordinay and hodge-witt} for the definition of ordinarity and Hodge-Wittness. 
A K3 surface is Hodge-Witt if and only if its Artin-Mazur height is finite, hence is quasi-$F$-split.

First of all, we have to notice that the scope of the (quasi-)$F$-splitting and ordinary/Hodge-Wittness is quite different.
Ordinary/Hodge-Wittness is defined for smooth proper schemes over a perfect field $k$ of characteristic $p>0$ in terms of the de Rham-Witt complex.
Within smooth proper schemes, “generic” one is ordinary.
For example, a generic smooth complete intersection in a projective space is ordinary(\cite{illusie-ordinary}). 

On the other hand, (quasi-)$F$-splitting is defined for $\F_p$-schemes, without smoothness or properness assumptions.
In fact, $F$-splitting has been studied in the theory of singularities or commutative algebras and nowadays it is a key notion in the theory of $F$-singularities. 
In this context, it is also called \textit{$F$-pure}.
Also, when a smooth projective scheme is quasi-$F$-split, it has non-positive Kodaira dimension (\cite[Proposition 3.14]{kttwyy1}).

Despite these differences, we investigate an analogy between quasi-$F$-split/$F$-split and Hodge-Witt/ordinary.
Besides the K3 case mentioned as above, we find that abelian varieties supports this analogy:
\begin{theorem}[Theorem \ref{quasi f split abelian variety}]
    For abelian varieties in positive characteristic, quasi-$F$-split is equivalent to Hodge-Witt and $F$-split is equivalent to ordinary.
\end{theorem}
\noindent
Remark that some rational varieties and Enriques surfaces break such a coincidence (see \S \ref{nonexamples}).

What surprises us is a similarity of the behaviour of these notions under taking products.
Recall that the following theorem due to Ekedahl:
\begin{theorem}\cite[p.91 Proposition 2.1, p.97 Proposition 7.2]{ekedahl-multiplicative-ii}\label{ordinary hodge witt}
Let $X$ and $Y$ be smooth proper schemes over a perfect field $k$ of characteristic $p>0$.
Then we have the following;
\begin{enumerate}
    \item if $X$ is ordinary and $Y$ is Hodge-Witt, then $X\times Y$ is Hodge-Witt,
    \item if $X\times Y$ is Hodge-Witt, then one of the factors is ordinary and the other is Hodge-Witt.
\end{enumerate}
\end{theorem}
Note that this type of phenomena (especially that of (2)) are observed from the very beginning of the invention of $p$-adic cohomology theory \cite{Serre58}\cite{serre58av}:
An ordinary elliptic curve (i.e., its $p$-rank is one) is ordinary in the sense of Block-Kato and all smooth projective curve is Hodge-Witt.
Serre shows that the self product of a supersingular elliptic curve is not Hodge-Witt \cite[\S1, Corollaire]{serre58av}.
N. Katz revealed that this phenomena is ubiquitous by showing the above theorem when the crystalline cohomologies of $X$ and $Y$ are torsion free \cite{katz1983}.
Then Ekedahl proves the theorem in general using his K\"unneth theory for de Rham-Witt complex.
We will show that an analogous theorem holds between $F$-split schemes and quasi-$F$-split schemes:
\begin{theorem}[Theorem \ref{fsplit-times-quasifsplit}, Theorem \ref{double-nonfsplit-nonquasifsplit}]\label{f-split times quasi-f-split}
    Let $k$ be a perfect field of characteristic $p>0$.
    Let $X$ and $Y$ be $k$-schemes.
    Then the following assertions hold:
    \begin{enumerate}
        \item If $X$ is $F$-split and $Y$ is quasi-$F$-split, then $X \times_k Y$ is quasi-$F$-split.
        \item 
        Assume that $X$ and $Y$ satisfy one of the following conditions;
        \begin{enumerate}
        \item $X=\Spec(A)$ and $Y=\Spec(B)$ where $A$ and $B$ are $F$-finite noetherian $k$-algebras,
        \item $X$ and $Y$ are geometrically connected proper schemes over $k$.
        \end{enumerate}
        If $X \times Y$ is quasi-$F$-split, then one of the factors is $F$-split and the other is quasi-$F$-split.
    \end{enumerate}
\end{theorem}
Remark that a similar result is proved by Kawakami-Takamatsu-Yoshikawa using their Fedder type criterion for quasi-$F$-splitting \cite{kty22fedder}.
But the result in this paper holds with fewer assumptions and the proof is quite direct.
The proof of (1) depends on a construction of tensor products of Witt rings due to Antieau-Nikolaus \cite{antieau-nikolaus}.
The proof of (2) is surprisingly simple because it follows from the formula $V(x)V(y)=pV(xy)$ in Witt ring where $V$ denotes the Verschiebung.

As an immediate application of Theorem \ref{f-split times quasi-f-split}(2) to birational geometry, we give the following.
\begin{corollary}[Theorem \ref{fano-and-f-split}]
    For any $p$, there is a klt Fano fourfold over $\overline{\F}_p$ which is not quasi-$F$-split.
\end{corollary}
Note that, Watanabe showed any one-dimensional log Fano pair with standard coefficients is $F$-split when $p>5$ \cite{watanabe91f}.
This result played a key role in the proof of three-dimensional minimal model program in charactersitic $p>5$ by Hacon and Xu \cite{haconxu15}. 
Later, Cascini, Tanaka and Witaszek showed that Watanabe's result does not hold in dimension two \cite{CTW15a}.
They constructed klt del Pezzo surfaces which are not $F$-split in any positive characteristic.
The notion of quasi-$F$-splitting remedied this situation somehow.
In fact, log Fano curves are quasi-$F$-split in any characteristic \cite[Corollary 5.16]{kttwyy1} and log del Pezzo surfaces are quasi-$F$s-split if $p>5$ \cite[Theorem B]{kttwyy1}.
But Theorem \ref{f-split times quasi-f-split}(2) allows us to conclude that the self product of the Cascini-Tanaka-Witaszek's non $F$-split del Pezzo surfaces, which are in fact four-dimensional klt Fano varieties, are not quasi-$F$-split.

The same self product trick and Theorem \ref{f-split times quasi-f-split}(2) give the following:
\begin{corollary}[Theorem \ref{dense-quasi-f-pure}, Theorem \ref{quasi-f-split-hilbert}]
\begin{enumerate}
    \item Dense quasi-$F$-pure conjecture (see \S \ref{dense-quasi-f-pure-section} for its definition) implies the dense $F$-pure conjecture.
    \item Let $X$ be a smooth projective scheme over a perfect field $k$.
        Assume that $X$ is not $F$-split, then the Hilbert scheme of points of $X$ is not quasi-$F$-split.
\end{enumerate}    
\end{corollary}

The first statement is an analogue of Joshi's theorem \cite{joshi06hodgewittreduction} on Hodge-Witt reduction of smooth projective varieties over a number field.

The Kumar-Thomsen's theorem says the Hilbert scheme of points of an $F$-split smooth projective surface is also $F$-split \cite{kumarthomsen}.
The second statement says a quasi-$F$-split analogue of Kumar-Thomsen's theorem does not hold.

Throughout this paper, $k$ denotes a perfect field of characteristic $p>0$ and $W$ is the ring of Witt vectors of $k$ and $\sigma$ is the Frobenius of $W$.
For a scheme $X$ of finite type over $k$, we denote by $\omega_X$ the dualizing sheaf of $X$.


\subsection*{Acknowledgement}
The author thanks to Hiromu Tanaka and Jakub Witaszek for kindly teaching him birational geometry.
He is also grateful to Tatsuro Kawakami and Teppei Takamatsu and Sho Yoshikawa for sending him the preprint \cite{kty22fedder}.
He was supported by JSPS KAKENHI Grant number JP19K14501.

\section{Preliminaries}
In this section, we recall what are Hodge-Witt and ordinary schemes and their relation to (quasi-)$F$-splitting.

Let $X$ be a smooth proper scheme over $k$.
We have the associated crystalline cohomology $\Hcris^*(X/W)$, which is a finitely generated $W$-module.
Also we have the de Rham-Witt complex $W\Omega_X^{\bullet}$, which computes the crystalline cohomology \cite{illusie_de_rham_witt}.
Each degree of the de Rham-Witt complex has two operators
\[
F \colon W\Omega_X^i \to W\Omega_X^i, \\
V \colon W\Omega_X^i \to W\Omega_X^i
\]
satisfying $FV=p=VF$ and $FdV=d$.
Furthermore, each terms is a projective limit $W\Omega_X^i=\lim_nW_n\Omega_X^i$ along the morphism $R \colon W_{n+1}\Omega_X^i \to W_n\Omega_X^i$.
\begin{definition}\label{ordinay and hodge-witt}
    Let $X$ be a smooth proper scheme over $k$.
    We say $X$ is \textit{ordinary} if the Frobenius action $F$ on each cohomology group $H^j(X, W\Omega_X^i)$ is bijective for any $i, j$.
    We say $X$ is \textit{Hodge-Witt} if each $H^j(X, W\Omega_X^i)$ is finitely generated as a $W$-module.
\end{definition}
A good survey of de Rham-Witt cohomology including these notions and Theorem \ref{ordinary hodge witt} is \cite{illusie-ekedahl}.
\begin{remark}
    \begin{enumerate}
        \item If $X$ is ordinary, then it is Hodge-Witt (c.f. Lemma \ref{F cobounded implies  finiteness}).
        \item A generic smooth complete intersection in a projective space of any degree is ordinary.
        \cite{illusie-ordinary}
    \end{enumerate}
\end{remark}

\begin{example}
Ordinary elliptic curves(i.e., those have a non-trivial $p$-torsion $\Bar{k}$-point) are ordinary. 
Supersingular elliptic curves are not ordinary but Hodge-Witt.
The product of two supersingular elliptic curves is not Hodge-Witt.
\end{example}

Let $X$ be an $\F_p$-scheme.
The Frobenius morphism $F \colon W_n\O_X \to F_*W_n\O_X$ induces 
\[
\begin{tikzcd}
    W_n\O_X \arrow{r}{F} \arrow[d, twoheadrightarrow, "R^{n-1}"'] & F_*W_n\O_X \arrow[d, twoheadrightarrow] \\
\O_X \arrow{r}{F} & F_*(W_n\O_X/p)
\end{tikzcd}
\]
where the lower $F$ is given by $F(x)= [x]^p ~\text{mod}~ p$ for $x \in \O_X$.
Here $[-] \colon \O_X \to W_n\O_X$ is the Teichm\"{u}ller lift.
Note that the above diagram is a pushout diagram.
We sometimes use a notation $\overline{W}_n\O_X:=W_n\O_X/p$.

\begin{definition}
    Let $X$ be an $\F_p$-scheme and let $n$ be a positive integer.
    We say $X$ is \textit{$n$-quasi-$F$-split} if there exists a dashed arrow in the following diagram of $W_n\O_X$-modules which makes the diagram commutes; 
\[\begin{tikzcd}
W_n\O_X \arrow{r}{{F}} \arrow{d}{R^{n-1}} & F_* W_n \arrow[dashed]{ld}{\exists} \O_X \\
\O_X.
\end{tikzcd}
\] 
We say $X$ is \textit{quasi-$F$-split} if it is $n$-quasi-$F$-split for some $n$ and the least of such $n$ is called the \textit{quasi-$F$-split height} of $X$, denoted by $\sht(X)$.
\end{definition}
Clearly $X$ is $n$-quasi-$F$-split if and only if $F \colon \O_X \to F_*(W_n\O_X/p)$ splits as a morphism of $\O_X$-modules.
We emphasize that this notion is interesting not only for global varieties but also local varieties (singularities) and quasi-$F$-splitting implies very strong restrictions on varieties.
For example, we have the following.
\begin{proposition}
    If a smooth projective scheme over $k$ is quasi-$F$-split, then its Kodaira dimension is non-positive.
\end{proposition}
\begin{proof}
    See \cite[Proposition 3.14]{kttwyy1}.
\end{proof}

We has a following criterion of $n$-quasi-$F$-splitness:
\begin{proposition}\label{quasi-f-split criterion}
    Let $X$ be a geometrically connected proper scheme of dimension $d$ over $k$.
    Then $X$ is $n$-quasi-$F$-split if and only if the morphism
    \[
    F \colon H^d(X, \omega_X) \to H^d(X, \omega_X\otimes_{\O_X}F_*(W_n\O_X/p))
    \]
    induced by $id_{\omega_X}\otimes F$ is injective.
\end{proposition}

\begin{proof}
    By definition, $n$-quasi-$F$-splitting is equivalent to the evaluation map
    \[
    H^0(X, \cHom_{\O_X}(F_*W_n\O_X/p, \O_X)) \to H^0(X, \O_X)=k
    \]
    being surjective.
    By the Serre duality, we get the desired statement.
\end{proof}


A direct connection between $F$-split/quasi-$F$-split and ordinary/Hodge-Witt is the following due to Nakkajima \cite{nakkajima2022}.
\begin{theorem}\label{nakkajima finite}
Let $X$ be a proper scheme over $k$.
\begin{enumerate}
    \item If $X$ is $F$-split, then the Frobenius $F$ on $H^i(X, W\O_X)$ is bijective for every $i$.
    \item If $X$ is quasi-$F$-split, then $H^i(X, W\O_X)$ is finitely generated for every $i$.
\end{enumerate}
\end{theorem}

\begin{proof}
(1)It is enough to prove that $F$ on $H^i(X, W_n\O_X)$ is bijective for every $n$.
The Frobenius $F$ on $H^i(X, \O_X)$ is injective by the assumption and hence bijective.
By induction on $n$, it is easy to show the desired statement.

(2) This is proved in \cite[Theorem 4.3, Remark 4.5]{nakkajima2022}, but we repeat the argument here for  reader's convenience. 

    First recall that the following lemma.

    \begin{lemma}\cite[\S5, Proposition 3]{Serre58}\label{F cobounded implies  finiteness}
        Let $M$ be a $W$-module equipped with a $\sigma$-linear operator $F \colon M \to M$ and a $\sigma^{-1}$-linear operator $V \colon M \to M$ satisfying $FV=p=VF$.
        Assume that $M$ is $V$-profinite, i.e., $M \simeq \varprojlim_nM/V^nM$ and $M/VM$ is of finite length.
        If $M/FM$ is of finite length, then $M$ is a finitely generated $W$-module.
    \end{lemma}
    An immediate consequence of this lemma is the following:
    \begin{corollary}\cite[\S5, Corollaire 1]{Serre58}
        Let $X$ be a proper scheme over $k$ and let $i$ be an integer.
        If $\varprojlim H^i(X, W_m\O_X/F)$ is of finite length, then $H^i(X, W\O_X)$ is finitely generated.
    \end{corollary}
    From this, it is enough to show that quasi-$F$-splitting implies the boundedness of the sequence $\{h^i(X, W_m\O_X/F)\}_m$ for each $i$.
    Consider the following two exact sequences;
    \begin{align*}
        0 \to \O_X \xrightarrow{F} W_m\O_X/p \to W_m\O_X/F \to 0, \\
        0 \to W_{m-1}\O_X/F \xrightarrow{V} W_m\O_X/p \xrightarrow{R^{m-1}} \O_X \to 0.
    \end{align*}
    Assume that $m \geq \sht(X)$.
    Then the first sequence splits, and hence we have an equality
    \[
    h^i(X, W_m\O_X/F)=h^i(X, W_m\O_X/p)-h^i(X, \O_X).
    \]
    On the other hand, the second sequence implies an inequality
    \[
    h^i(X, W_m\O_X/p) \leq h^i(X, W_{m-1}\O_X/F)+h^i(X, \O_X).
    \]
    Combining these two, we get
    \[
    h^i(X, W_m\O_X/F)\leq h^i(X, W_{m-1}\O_X/F)
    \]
    provided $m \geq \sht(X)$.
\end{proof}



\section{Examples and nonexamples}\label{nonexamples}
In this section, we collect some (non)examples of (quasi-)$F$-split schemes and Hodge-Witt or ordinary schemes within smooth proper schemes over $k$.
\begin{theorem}
    Let $X$ be a K3 surface or an abelian variety of dimension $g$ over $k$.
    Then $X$ is quasi-$F$-split if and only if it is Hodge-Witt.
\end{theorem}

\begin{proof}
    We may assume that $k$ is an algebraically closed field since a perfect base change does not change the quasi-$F$-split height (\cite[Proposition 3.4]{yobuko20}).
    By \cite[p. 653, \S7.2.]{illusie_de_rham_witt}, a $K3$ surface is Hodge-Witt if and only if it has finite Artin-Mazur height.
    Hence the $K3$ cases follows from \cite[Theorem 4.5]{yobuko19}.

    By \cite[Corollary 6.3.16]{illusie-ekedahl}, an abelian variety of dimension $g$ is Hodge-Witt if and only if its $p$-rank ($:=$ the $\F_p$ -dimension of the abelian group of $p$-torsion $\Bar{k}$-points of the abelian variety) is $g$ or $g-1$.
    Hence the abelian variety case follows from Theorem \ref{quasi f split abelian variety}.
\end{proof}

\begin{theorem}\label{quasi f split abelian variety}
Let $A$ be a $g$-dimensional abelian variety over $k$.
Then the quasi-F-split height of $A$ is given by
\[
\sht(A)
=
\begin{cases}
1 ~~\text{if}~~ f(A)=g,\\
2 ~~\text{if}~~ f(A)=g-1, \\
\infty ~~\text{if}~~ f(A) \leq g-2
\end{cases}
\]
where $f(A)$ is the $p$-rank of $A$.
In particular, an abelian variety is quasi-$F$-split if and only if Hodge-Witt.
\end{theorem}

\begin{proof}
We may assume $k$ is an algebraically closed field.
The equivalence between $f(A)=g$ and $\sht(A)=1$ is well known.

Assume that $f(A)=g-1$.
Since $\Hcr^g(A/W)=\wedge^g\Hcr^1(A/W)$, we see that the slope less than one part $\Hcr^g(A/K)_{[0, 1)}$ is two dimensional over $K:=W[\frac{1}{p}]$.
We claim that $H:=H^g(A, W\O_A)$ is free $W$-module of rank two.
Note that, a priori, $H$ is not necessarily finitely generated as a $W$-module, but equipped with operators $F, V$ such that $H$ is $V$-adically complete.
Let $M:=H/H_{tor}$ be the quotient of $H$ by the torsion submodule $H_{tor}$.
Then $M$ is a free $W(k)$-module of rank two, equipped with operators $F$ and $V$.
The exact sequence $0 \to H_{tor} \to H \to M \to 0$ induces an exact sequence
\[
0 \to H_{tor}/VH_{tor} \to H/VH \to M/VM \to 0
\]
since $V$ on $M$ is injective.
We know that $H/VH=H^g(A, \O_A)$ is one dimensional and $M/VM$ is nonzero since any slope of the Dieudonn\'e module $M$ is less than $1$.
By the above exact sequence, we see $H_{tor}/VH_{tor}=0$.
Since $H_{tor}$ is a submodule of a $V$-adically complete module $H$, it is $V$-adically separated.
So we see $H_{tor}=0$.

Recall that one dimensional Dieudonn\'e crystals over an algebraically closed field are classified by their height.
In particular, there are elements $e_1, e_2 \in M$ which form a basis of $H$ and satisfy the relations
\begin{align*}
    V(e_1)=e_2, V(e_2)=pe_1, \\
    F(e_1)=e_2, F(e_2)=pe_1.
\end{align*}
Hence $F$ on $H$ induces an injection
\[
F \colon H/VH \to H/V^2H. 
\]
This map fits into
\[
\xymatrix{
H^g(A, \O_A) \ar[r]^-{F}\ar[d]_-{\simeq} & H^g(A, \overline{W}_2\O_A) \ar[d] \\
H/VH \ar[r]_-{F} & H/V^2H
}
\]
and we see that the top horizontal map is not zero.
This means $\sht(A) \leq 2$ by Proposition \ref{quasi-f-split criterion}.

Now assume that $f(A) \leq g-2$.
Then $\Hcr^g(A/K)_{[0, 1)}$ is zero and $H^g(A, W\O_A)$ is a torsion module.
If $A$ is quasi-$F$-split, then $H^g(A, W\O_A)$ is finitely generated by Theorem \ref{nakkajima finite}, and hence, is of finite length.
But, by \cite[Th\'eor\`eme 2]{serre58av}, $R \colon H^i(A, W_{n+1}\O_A) \to H^i(A, W_n\O_A)$ is surjective for any $i$, so the length of $H^g(A, W_n\O_A)=n$.
This is a contradiction.
\end{proof}

We remark that, even for smooth projective schemes which have non-positive Kodaira dimension except abelian varieties and K3 surfaces, $F$-split/quasi-$F$-split and ordinary/Hodge-Witt have very different features:
\begin{enumerate}
    \item There is a rational threefold which is ordinary but not quasi-$F$-split.
    Indeed, let $Y$ be the blow up of $\P^3$ at all $\F_p$-rational points $\P^3(\F_p)$.
    Then blowup all the strict transforms of the $\F_p$-rational lines and we get a rational threefold $X$.
    Achinger and Zdanowicz show that this $X$ is ordinary but not liftable to $W_2(k)$ \cite[Theorem 4.1]{achinger-zdanowicz17}.
    In particular, $X$ is not quasi-$F$-split by \cite[Theorem 4.4]{yobuko19}.
    \item There is a rational fourfold which is $F$-split but not Hodge-Witt.
    Indeed, let $Y \subset \P^3$ be a supersingular K3 surface.
    Consider $\P^3$ as a subscheme of $\P^4$ which is the zero locus of the first homogeneous coordinate.
    Let $X$ be the blowup of $\P^4$ with the center $Y$.
    Then $X$ is $F$-split but not Hodge-Witt by \cite[Theorem 5.5]{joshi-rajan03}.
    \item Enriques surfaces are Hodge-Witt for any prime $p>0$ by \cite[p.656, Proposition 7.3.6]{illusie_de_rham_witt}.
    When $p>2$, consider its K3 covering, then the original Enriques surface is quasi-$F$-split if and only if the K3 cover is not supersingular \cite{yobuko20}.
    Note that Enriques surfaces are ordinary when $p>2$.

    When $p=2$, Enriques surfaces are divided into three types; classical, singular, and supersingular according to \cite[\S3]{bombieri-mumford76}.
    Classical or superisngular Enriques surfaces are not quasi-$F$-split and singular ones are $F$-split \cite[Theorem 5.8]{yobuko20}.
    Note that singular or supersingular Enriques surfaces have trivial canonical bundle, not just numerically trivial.
    \begin{table}[h]
    \centering
    \begin{tabular}{|c||c|c|}\hline
       classical & ordinary & not quasi-$F$-split \\ \hline
       singular & ordinary & $F$-split  \\ \hline
       supersingular & not ordinary but Hodge-Witt & not quasi-$F$-split \\ \hline
    \end{tabular}
    \label{tab:my_label}
\end{table}
\end{enumerate}

Let us remark (non)ordinarity of an Enriques surface in characteristic $2$.
By \cite{illusie_de_rham_witt}, we know that the slope spectral sequence of an Enriques surface is $E_1$-degenerate and its $E_1$-terms are given as follows:((i)=ordinary, (ii)=singular, (iii)=supersingular)
\settasks{counter-format=(\roman*)}
\settasks{label-width=18pt}
\begin{tasks}(3)
\task
$\begin{matrix}
    0 & k & W \\
    0 & W^{10}\oplus k & 0\\
    W & 0 & 0
\end{matrix}
$
\task
$
\begin{matrix}
    k & 0 & W \\
    0 & W^{10} & k\\
    W & 0 & 0
\end{matrix}
$
\task
$
\begin{matrix}
    k & k & W \\
    0 & W^{10} & 0\\
    W & 0 & 0
\end{matrix}
$
\end{tasks}
We will determine which type is ordinary or not.
Let us consider the classical case.
The only non-trivial Frobenius action is $F$ on $H^2(X, W\Omega_X^1) \simeq H^2(X, \Omega_X^1) \simeq k$.
By the Serre duality, this is dual to $H^0(X, \Omega_X^1)$ with the action given by the Cartier operator $C$.
Note that the Hodge to de Rham spectral sequence degenerates at $E_1$-terms.
By \cite{katsura82}, we know that $X$ has an elliptic or a quasi-elliptic fibration $f \colon X \to \P^1$ with an affine parameter $t$ of $\P^1$ such that $f^*(\frac{dt}{t})$ is a non-zero regular one form on $X$.
This implies that Cartier operator on $H^0(X, \Omega_X^1)$ is bijective.
The singular case follows from the definition.
Supersingular Enriques are not orinary since the Frobenius action on $H^2(X, W\O_X)\simeq H^2(X, \O_X) \simeq k$ is zero by definition.
We remark that the Frobenius action on $H^2(X, W\Omega_X^1) \simeq k$ is also zero since this cohomology is spanned by $dt$ for some elliptic or quasi elliptic fibration $X \to \P^1$ as in the classical case.



\section{$F$-split analogue of Ekedahl's theorem}
In this section, we prove an $F$-split analogue of Ekedahl's theorem \ref{ordinary hodge witt}. 

To begin with, we review a construction of tensor product of Witt rings \cite[Theorem 4.16]{antieau-nikolaus}.
Here we only present what we need later, in particular we work over the base ring $W$, though the original work is over $\Z$.

\begin{definition}
A \textit{$p$-typical Cartier module} over $k$ is a triple $(M, F, V)$ where $M$ is a $W$-module and $F$ (resp. $V$) is a $\sigma$-linear (resp. $\sigma^{-1}$-linear) endomorphism of $M$ satisfying $FV=p$.
\end{definition}

For two $p$-typical Cartier module over $k$, $M$ and $N$, we will define their tensor product $M \boxtimes_{W} N$, which is a new $p$-typical Cartier module over $k$.

Let $M$ be a $W$-module equipped with a $\sigma$-linear map $F \colon M \to M$.
We define a $p$-typical Cartier module $M[V]$ over $k$ by 
\[
M[V]:=\bigoplus_{i \geq 0}\sigma_*^iMV^i,
\]
where $V^i$'s are new symbols (so $\sigma_*^iMV^i$ is isomorphic to the $W$-module $\sigma_*^iM$).
The $(F, V)$-structure on $M[V]$ are given by
\begin{align*}
&F(\sum_{i \geq 0} m_iV^i)=F(m_0)+\sum_{i>0}pm_iV^{i-1}, \\
&V(\sum_{i \geq 0} m_iV^i)= \sum_{i \geq 0} m_iV^{i+1}.
\end{align*}
Note that $F$ (resp. $V$) on $M[V]$ is $\sigma$-linear (resp. $\sigma^{-1}$-linear) thanks to the twisting $\sigma_*^i$ on degree $i$ and $F$ and $V$ clearly satisfy $FV=p$.

For two $p$-typical Cartier module $M$ and $N$ over $k$, we define a $p$-typical Cartier module $M \boxtimes_{W} V$ to be
\[
M \boxtimes_{W} N:=(M \otimes_{W}N)[V]/ \sim
\]
where the equivalence relation $\sim$ is generated by
\[
(m \otimes Vn)V^k \sim (Fm \otimes n)V^{k+1}, ~~~~(Vm \otimes n)V^k \sim (m \otimes Fn)V^{k+1}
\]
for all $m \in M$, $n \in N$ and $k \geq 0$.
The $(F, V)$-structure on $M \boxtimes_{W} N$ is induced by the one on $(M \otimes N)[V]$.

Finally we define $M \widehat{\boxtimes}_{W}N$ to be the $V$-adic completion of $M \boxtimes_{W} N$.

\begin{theorem}\cite[Theorem 4.16]{antieau-nikolaus}\label{monoidal structure on witt}
Let $k$ be a perfect field.
For any pair of algebras $A$ and $B$ over $W=W(k)$, we have a natural isomorphism
\[
W(A)\widehat{\boxtimes}_{W} W(B) \simeq W(A \otimes_{W} B).
\]
In particular, if $A$ and $B$ are $k$-algebras, then $W(A)\widehat{\boxtimes}_{W} W(B) \simeq W(A \otimes_k B)$.
\end{theorem}
\begin{remark}
    In \cite{antieau-nikolaus} they prove this theorem (over $\Z$) for non-commutative rings using non-commutative Witt vectors.
\end{remark}
\begin{proof}
We reproduce the proof for reader's convenience.
Clearly we have a functorial map $W(A)\widehat{\boxtimes}_{W} W(B) \to W(A \otimes_{W} B)$ and we want to show this is an isomorphism.
\begin{claim}
We may assume $A \simeq W[M]$ and $B \simeq W[N]$ for some commutative monoids $M$ and $N$.
\end{claim}
\begin{proof}[proof of Claim]
Take a split coequalizer diagram
\[
\xymatrix{
W[M_1] \ar@<1.2ex>[r]^-{d_0} \ar@<-1.2ex>[r]_-{d_1} & W[M_0] \ar@<-0.2ex>[l]^-{s}\ar[r]^-{\pi} & A\\
}
\]
where $W[M_i]$ is the $W$-algebra freely generated by a commutative monoid $M_i$.
Here a split coequalizer diagram means a tuple of arrows as above satisfying relations $\pi \circ d_0=\pi \circ d_1, d_0\circ s=id =d_1 \circ s$ and $\pi$ being surjective.
Then the diagram applied term-wise Witt functor
\[
\xymatrix{
W(W[M_1]) \ar@<0.7ex>[r] \ar@<-0.7ex>[r] & W(W[M_0]) \ar[l]\ar[r] & W(A)\\
}
\]
is also a split coequalizer diagram.
Similarly, let
\[
\xymatrix{
W[N_1] \ar@<0.7ex>[r] \ar@<-0.7ex>[r] & W[N_0] \ar[l]\ar[r] & B\\
}
\]
be a split coequalizer diagram for $B$.
By taking term-wise tensor products, we get a split coequalizer diagram for $A\otimes_k B$ and then one for $W(A\otimes_k B)$:
\[
\xymatrix{
W(W[M_1]\otimes W[N_1]) \ar@<0.7ex>[r] \ar@<-0.7ex>[r] & W(W[M_0]\otimes W[N_0]) \ar[l]\ar[r] & W(A\otimes B).\\
}
\]
Similarly, by taking tensor product of split coequalizer diagram for $W(A)$ and $W(B)$, we get
\[
\xymatrix{
W(W[M_1])\otimes_{W}W(W[N_1]) \ar@<0.7ex>[r] \ar@<-0.7ex>[r] & W(W[M_0])\otimes_{W}W(W[N_0]) \ar[l]\ar[r] & W(A)\otimes_{W}W(B).
}
\]
By definition, $M \widehat{\boxtimes}_W N$ fits into an exact sequence
\[
 ((M\otimes_W N)[[V]]) \oplus  ((M\otimes_W N)[[V]]) \to (M\otimes_W N)[[V]] \to M \widehat{\boxtimes}_W N \to 0.
\]
Hence we have an exact sequence
\[
W(W[M_1])\widehat{\boxtimes}_WW(W[N_1]) \to W(W[M_0])\widehat{\boxtimes}_WW(W[N_0]) \to W(A)\widehat{\boxtimes}_WW(B) \to 0.
\]
So the problem is reduced to show the equality
\[
W(W[M_i])\widehat{\boxtimes}_WW(W[N_i]) = W(W[M_i]\otimes_{W}W[N_i]). 
\]
\end{proof}
For $A=W(k)[M]$, we claim that $W(A) \simeq A[[V]]$.
Consider the monoid map $M \to W(A)$ which maps an element of $M$ to its Teichm\"uller lift in $W(A)$.
This induces an algebra morphism $A=W[M] \to W(A)$, which is compatible with Frobenius.
Here we put a Frobenius structure on $A$ induced by the multiplication by $p$ map on $M$ and $\sigma$ on $W$.
By the compatibility with $F$, this map gives us a map of $p$-typical Cartier modules
\[
A[[V]] \to W(A),
\]
which reduces to the identity of $A$ after modulo $V$.
Since both side is $V$-adically complete, we get the desired isomorphism $A[[V]] \simeq W(A)$.
It remains to prove
\[
A[[V]] \widehat{\boxtimes} B[[V]] \simeq (A \otimes B)[[V]].
\]
This follows from the fact that the morphism
\[
A[V] \boxtimes B[V] \to (A\otimes B)[V]
\]
is an isomorphism.
(The existing of inverse map follows from the construction of $M \boxtimes N$.)
\end{proof}

\begin{theorem}\label{fsplit-times-quasifsplit}
    Let $k$ be a perfect field of characteristic $p>0$.
    Let $X$ and $Y$ be schemes over $k$.
    If $X$ is $F$-split and $Y$ is $n$-quasi-$F$-split, then $X \times _k Y$ is $n$-quasi-$F$-split.
\end{theorem}
\begin{proof}
    We first consider affine case; $X=\Spec(A)$ and $Y=\Spec (B)$ for $k$-algebras $A, B$.
    Let $\sigma_A \colon F_*A \to A$ be a splitting section of $A$ and let $\sigma_B \colon F_*W_nB \to B$ be a quasi-splitting section of $B$.

    Since $W(A) \widehat{\boxtimes} W(B)/V^n = W(A)\boxtimes W(B)/V^n$, it is enough to construct a morphism
\[
\sigma \colon W(A)\otimes W(B)[V] \to A \otimes B
\]
which satisfies
\begin{align*}
   & \sigma ((a \otimes Vb)V^k)=\sigma ((Fa \otimes b)V^{k+1}), \\
    &\sigma ((Va \otimes b)V^k)=\sigma ((a \otimes Fb)V^{k+1}),\\
    &\sigma ((a \otimes b)V^{m})=0
\end{align*}
for any $a \in W(A)$, $b \in W(B)$, $k \geq 0$ and $m > n$.
We set
\[
\sigma ((a \otimes b)V^k):=\sigma_A^{k+1}(Ra) \otimes \sigma_B(V^kb).
\]
Here we denote the composition $W(B) \to W_n(B) \xrightarrow{\sigma_B} B$ by the same symbol $\sigma_B$.
It is straight to check this $\sigma$ satisfies all the desired properties.

For general $k$-schemes $X, Y$, we start with global splittings $\sigma_X \colon F_*\O_X \to \O_X$ and $\sigma_Y \colon F_*W_n\O_Y \to \O_Y$.
Take affine open coverings $X=\bigcup_iU_i$ and $Y=\bigcup_j V_j$.
Then the local affine splittings $\sigma_{U_i\times V_j}$ constructed as above clearly glue to a global splitting $\sigma_{X \times Y}\colon F_*W_n\O_{X\times Y} \to \O_{X\times Y}$.
\end{proof}

For the converse direction, we prove a local statement and a global one.
\begin{theorem}\label{factors of quasi f split}
    Let $k$ be a perfect field of characteristic $p>0$ and let $X, Y$ be schemes over $k$.
    Assume that $X$ and $Y$ satisfy one of the following conditions;
    \begin{enumerate}
        \item $X=\Spec(A)$ and $Y=\Spec(B)$ where $A$ and $B$ are $F$-finite noetherian $k$-algebras,
        \item $X$ and $Y$ are geometrically connected proper schemes over $k$.
    \end{enumerate}
    If $X \times_k Y$ is quasi-$F$-split, then one of the factors is $F$-split and the other is quasi-$F$-split.
\end{theorem}

\begin{proof}
Since the canonical maps $\O_X \to pr_{1*}\O_{X \times Y}$ and $\O_Y \to pr_{2*}\O_{X\times Y}$ split, a quasi-$F$-splitting of $X\times Y$ implies the ones of $X$ and $Y$ by \cite[Proposition 3.4 (1)]{yobuko20}.
Hence it is enough to show Lemma \ref{double-nonfsplit-nonquasifsplit}
\end{proof}

\begin{lemma}\label{double-nonfsplit-nonquasifsplit}
    Assume that $X$ and $Y$ satisfy one of the following conditions;
    \begin{enumerate}
        \item $X=\Spec(A)$ and $Y=\Spec(B)$ where $A$ and $B$ are $F$-finite noetherian $k$-algebras,
        \item $X$ and $Y$ are geometrically connected proper schemes over $k$.
    \end{enumerate}
    If both $X$ and $Y$ are not $F$-split, then $X \times_k Y$ is not quasi-$F$-split.
\end{lemma}

\begin{proof}
    Fix an integer $n >0$ and we will prove that $X \times Y$ is not $n$-quasi-$F$-split.
    First consider the case (1).
    We want to show $F \colon A\otimes_k B \to F_*\overline{W}_n(A \otimes_k B)$ does not split.
Note that, under the assumption, splitting is equivalent to purity (:=injective after tensoring any module) by \cite[Corollary 5.2]{hochsterroberts76}.

Since $A$ is not $F$-split, there is an $A$-module $M_A$ such that $F \colon M_A \to M_A \otimes_A (F_*A)$ is not injective.
Note that this map factors through as
\[
M_A \xrightarrow{F} M_A \otimes_A (F_*\overline{W}_n(A)) \xrightarrow{id \otimes R} M_A \otimes_A (F_*A).
\]
Since we have an exact sequence
\[
M_A \otimes_A (F_*^2\overline{W}_{n-1}(A)) \xrightarrow{id \otimes V} M_A \otimes_A (F_*\overline{W}_n(A))\xrightarrow{id \otimes R} M_A \otimes_A (F_*A) \to 0,
\]
there is a nonzero element $x_A \in M_A$ such that
\[
F(x_A)= \sum_i x_i \otimes V(a_i) \in M_A \otimes_A (F_*\overline{W}_n(A))
\]
for some $x_i \in M_A$ and $a_i \in \overline{W}_{n-1}(A)$.
Similarly, there is a $B$-module $M_B$ and a nonzero element $y \in M_B$ such that
\[
F(y)= \sum_jy_j \otimes V(b_j) \in M_B \otimes_B (F_*\overline{W}_n(B))
\]
for some $y_j \in M_B$ and $b_j \in \overline{W}_{n-1}(B)$.

Consider an $A\otimes_k B$-module $M_A \otimes_k M_B$ and its Frobenius $F \colon M_A \otimes_k M_B \to (M_A \otimes_k M_B)\otimes_{A\otimes B}F_*\overline{W}_n(A\otimes_k B)$.
This map factors as
\[
\xymatrix{
M_A \otimes_k M_B \ar[r]^-{F} \ar[rd]_-{F\otimes F}&(M_A \otimes_k M_B)\otimes_{A\otimes B}F_*\overline{W}_n(A\otimes_k B)\\
& (M_A\otimes_AF_*\overline{W}_n(A))\otimes_k(M_B\otimes_BF_*\overline{W}_n(B)).\ar[u]
}
\]
Clearly $x \otimes y \in M_A\otimes M_B$ is nonzero and this element is mapped as
\begin{align*}
F(x\otimes y)&=F(x)\otimes F(y)\\
&=(\sum_i x_i \otimes V(a_i))\otimes (\sum_jy_j \otimes V(b_j))\\
&=\sum_{i, j} (x_i\otimes y_j)\otimes V(a_i)V(b_j)\\
&=\sum_{i, j} (x_i\otimes y_j)\otimes pV(a_ib_j)=0
\end{align*}
Hence the map $M_A \otimes_k M_B \to (M_A \otimes_k M_B)\otimes_{A\otimes B}F_*\overline{W}_n(A\otimes_k B)$ is not injective.

Next consider the case (2).
Let $d$ (resp. $e$) be the dimension of $X$ (resp. $Y$).
We will show that
\[
F \colon H^{d+e}(X \times Y, \omega_{X \times Y}) \to H^{d+e}(X \times Y, \omega_{X \times Y}\otimes F_*\overline{W}_n\O_{X \times Y})
\]
is not injective.
By the K\"{u}nneth formula, this map factors as
\[
H^d(\omega_X) \otimes_k H^e(\omega_Y) \to H^d(\omega_X \otimes F_*\overline{W}_n\O_X) \otimes_k H^e(\omega_Y \otimes F_*\overline{W}_n\O_Y) \to H^{d+e}(\omega_{X \times Y}\otimes F_*\overline{W}_n\O_{X \times Y})
\]
Let $\eta_X$ (resp. $\eta_Y$) be a generator of $H^d(\omega_X)$ (resp. $H^e(\omega_Y)$).
By the assumption, as in the affine case, $F(\eta_X) \in H^d(\omega_X\otimes F_*\overline{W}_n\O_X)$ (resp. $F(\eta_Y) \in H^e(\omega_Y \otimes F_*\overline{W}_n\O_Y)$) is in the image of $V \colon H^d(\omega_X\otimes F_*\overline{W}_n\O_X) \to H^d(\omega_X\otimes F_*^2\overline{W}_{n-1}\O_X)$ (resp. $V \colon H^e(\omega_Y\otimes F_*\overline{W}_n\O_Y) \to H^e(\omega_Y\otimes F_*^2\overline{W}_{n-1}\O_Y)$).

Now consider affine open coverings $X =\bigcup_iU_i$ and $Y=\bigcup_jV_j$.
Then $X \times Y$ has an induced affine covering $X \times Y = \bigcup_{(i, j)}U_i \times V_j$.
Then the class $F(\eta_X)$ is represented by a cocycle of the form
\[
\{\Tilde{\eta}_{X, i_0\cdots i_d}\}_{i_0\cdots i_d} \in \bigoplus_{i_0\cdots i_d}\Gamma (U_{i_0 \cdots i_d}, \omega_X \otimes F_*\overline{W}_n\O_X)
\]
with 
\[
\Tilde{\eta}_{X, i_0\cdots i_d}=\sum_{\alpha}\Tilde{\eta}_{X, i_0\cdots i_d, \alpha}\otimes V(f_{i_0\cdots i_d, \alpha}).
\]
where
\[
\Tilde{\eta}_{X, i_0\cdots i_d, \alpha} \in \Gamma(U_{i_0 \cdots i_d}, \omega_X)~~\text{and}~~f_{i_0\cdots i_d, \alpha} \in \Gamma (U_{i_0 \cdots i_d}, F_*^2\overline{W}_{n-1}\O_X).
\]
Similarly, $F(\eta_Y)$ is represented by a cocyle $\{\Tilde{\eta}_{Y, j_0\cdots j_e}\}$ of the form
\[
\Tilde{\eta}_{Y, j_0\cdots j_e}=\sum_{\beta}\Tilde{\eta}_{Y, j_0\cdots j_e, \beta}\otimes V(g_{j_0\cdots j_e, \beta}).
\]
The class $\eta_{X \times Y}=\eta_X \otimes \eta_Y$ is a generator of $H^{d+e}(\omega_{X \times Y})$ and then $F(\eta_{X\times Y})=F(\eta_X) F(\eta_Y)$ is represented by a $(d+e)$-cocycle (who lives in $\bigoplus_{(i_0, j_0) \cdots (i_{d+e}, j_{d+e})} \Gamma (U_{i_0 \cdots i_{d+e}}\times V_{j_0 \cdots j_{d+e}}, \omega_{X \times Y} \otimes F_*\overline{W}_n\O_{X \times Y} )$) whose value on $U_{i_0 \cdots i_{d+e}} \times V_{j_0 \cdots j_{d+e}}$ is 
\[
\Tilde{\eta}_{X, i_0\cdots i_d}\Tilde{\eta}_{Y, j_d \cdots j_{d+e}},
\]
which is $0$ modulo $p$ by exactly the same computation as in the affine case.
\end{proof}

\begin{remark}
    When $A$ is a local ring with the maximal ideal $\mathfrak{m}$ and the canonical module $\omega_A$, the local cohomology $H_{\mathfrak{m}}^{\dim A}(\omega_A)$ does the job of $M_A$ in the proof.
\end{remark}

\section{Applications}
In this section, we give applications of Theorem \ref{factors of quasi f split}.
\subsection{Klt Fano variety}
Recall the following relation between (log) Fano varieties and $F$-splitting.
\begin{theorem}\label{fano-and-f-split}
    \begin{enumerate}
        \item (\cite{watanabe91f}) Let k be an algebraically closed field of characteristic $p>5$.
    Let $(\P^1, \Delta)$ be a one dimensional log Fano pair with standard coefficients.
    Then $(\P^1, \Delta)$ is $F$-split.
        \item (\cite{CTW15a}) For any prime $p$, there is a klt del Pezzo surface over $\overline{\F}_p$ which is not $F$-split.
    \end{enumerate}
\end{theorem}
Note that when $p \in \{2, 3, 5\}$, there are non-$F$-split one dimensional log Fano pairs.
The notion of quasi-$F$-splitting remedies the situation:
\begin{theorem}Let k be an algebraically closed field of characteristic $p>0$.
\begin{enumerate}
    \item (\cite[Corollary 5.16]{kttwyy1})
    Let $(\P^1, \Delta)$ be a one dimensional log Fano pair.
    Then $(\P^1, \Delta)$ is quasi-$F$-split.
    \item (\cite[Theorem C]{kttwyy2}) Assume that $p >42$.
    Let $(X, \Delta)$ be a log del Pezzo pair with standard coefficients.
    Then $(X, \Delta)$ is quasi-$F$-split.
\end{enumerate}
\end{theorem}
\noindent
We do not need to assume $\Delta$ has standard coefficients in (1).
Now we show that, in higher dimension, the situation is as in Theorem \ref{fano-and-f-split}(2).
\begin{theorem}
    Let $k$ be an algebraically closed field of characteristic p.
    Then there exists an klt Fano fourfold which is not quasi-$F$-split.
\end{theorem}

\begin{proof}
    By Theorem \ref{fano-and-f-split}(2), we can take a klt del Pezzo over $k$ which is not $F$-split.
    By Theorem \ref{double-nonfsplit-nonquasifsplit}, the self product $X \times X$ is not quasi-$F$-split.
    Note that  $X \times X$ is $\Q$-factorial by \cite{boissiere-gabber-serman}.
\end{proof}

\subsection{Dense $F$-pure reduction}\label{dense-quasi-f-pure-section}

A similar trick gives us the following.
Given a variety or singularity $X$ defined over a field of characteristic zero, take a model defined over $\Z$ and then consider its mod $p$ reductions $X_p$ for various primes $p$.
One can ask a relationship between properties of $X$ and $X_p$'s.
We often consider whether $X$ is log terminal or log canonical or not.

\begin{definition}
    Let $R$ be a finitely generated normal domain over a field $K$ of characteristic zero.
    We say $R$ is \textit{of dense $F$-pure type}(resp. \textit{of dense quasi-$F$-pure type}) if there is a finitely generated $\Z$-subalgebra $A$ of $K$ and a finitely generated flat $A$-algebra $R_A$ such that $R_A\otimes_AK\simeq R$ and $R\otimes_Ak(s)$ is $F$-split (resp. quasi-$F$-split) for every closed point $s$ of a dense subset of $\Spec(A)$.
\end{definition}
\noindent
There is a notion of \textit{strongly-$F$-regular type} and it is known that $X$ is klt if and only if it is of strongly $F$-regular type \cite{takagi04}.

The following conjecture is due to Hara and K.-i. Watanabe.

\leftline{(DFP) Every log canonical singularities in characteristic zero are of dense $F$-pure type.}
\noindent
Using quasi-$F$-splitting, we can consider the following statement:

\leftline{(DQFP) Every log canonical singularities in characteristic zero are of dense quasi-$F$-pure type.}
\noindent
Clearly (DFP) implies (DQFP).
Theorem \ref{factors of quasi f split} implies the converse:
\begin{theorem}\label{dense-quasi-f-pure}
    (DQFP) impies (DFP).
\end{theorem}


\subsection{Hilbert scheme of points}

Recall that the theorem by Kumar-Thomsen that the Hilbert scheme of points of an $F$-split quasiprojective smooth surface is also $F$-split if $p>2$  \cite{kumarthomsen}.
The next result says a quasi-$F$-split analogue of this theorem does not hold.

\begin{theorem}\label{quasi-f-split-hilbert}
    Let $X$ be a smooth projective scheme over $k$ and let $n \geq 2$ be an integer.
    Assume that the Hilbert scheme $X^{[n]}$ of $X$ of points of length $n$ is quasi-F-split.
    Then $X$ is $F$ split. 
\end{theorem}
\begin{proof}
    Consider the Hilbert-Chow morphism
    \[
    f \colon X^{[n]} \to X^{(n)}
    \]
     where $X^{(n)}$ is the $n$-th symmetric product of $X$.
     We know that $f_*\O_{X^{[n]}}=\O_{X^{(n)}}$ because the Hilbert-Chow morphism is projective birational and $X^{(n)}$ is normal since it has only quotient singularities. 
     This implies that a quasi-$F$-splitting for $X^{[n]}$ descends to a splitting for $X^{(n)}$.
     Now consider the natural quotient map
     \[
     \pi \colon X^n \to X^{(n)}.
     \]
     Let $U \subset X^n$ be the complement of the “partial” diagonals.
     It is well known that the restriction of $\pi$ to $U$ is \'etale.
     Hence the quasi-$F$-splitting of $X^{(n)}$ can be pullbacked on $U$.
     Since the codimension of $U$ is equal to $\dim X \geq 2$, we can conclude $X^n$ is also quasi-$F$-split.
     By Theorem \ref{factors of quasi f split}, we see that $X$ is $F$-split.
\end{proof}

\bibliographystyle{amsalpha}
\bibliography{references}

\end{document}